\documentclass[11pt,reqno]{amsart}
\usepackage{amssymb}
\usepackage{amscd}
\usepackage[latin1]{inputenc}
\usepackage{amsmath}
\usepackage{tikz}
\usepackage{amsfonts}
\usepackage{faktor}
\usepackage{latexsym}
\usepackage{verbatim}
\usepackage{graphicx}
\usepackage{epsfig}

\usepackage{enumitem}

\usepackage[linktocpage=true]{hyperref}
\usepackage[capitalize]{cleveref}
\hypersetup{citecolor = black,colorlinks,linkcolor = black,urlcolor = black}

\newtheorem*{T1}{Theorem~\ref{HD}}

\newtheorem{theorem}{Theorem}[section]

\newtheorem{lemma}[theorem]{Lemma}
\newtheorem{question}[theorem]{Question}
\newtheorem{proposition}[theorem]{Proposition}
\newtheorem{corollary}[theorem]{Corollary}

\usepackage[hmargin=3cm,vmargin=3cm]{geometry}

\theoremstyle{definition}
\newtheorem{definition}[theorem]{Definition}
\newtheorem{example}[theorem]{Example}

\theoremstyle{plain}
\newcounter{MainTheoremCounter}

\begin{document}

\title[Periodic structure and Schr\"{o}dinger operators for codings of circle rotations]{Periodic structure and Schr\"{o}dinger operators for codings of circle rotations}

\author{Luke Hetzel}
\address{Luke Hetzel\\
Department of Mathematics\\
University of Denver\\
2390 S. York St.\\
Denver, CO 80208}
\email{luke.hetzel@du.edu}

\author{Ronnie Pavlov}
\address{Ronnie Pavlov\\
Department of Mathematics\\
University of Denver\\
2390 S. York St.\\
Denver, CO 80208}
\email{ronnie.pavlov@du.edu}
\urladdr{http://www.math.du.edu/$\sim$rpavlov/}

\thanks{The second author gratefully acknowledges the support of a Simons Foundation Collaboration Grant.}

\keywords{Schr\"{o}dinger operators, Gordon lemma, circle rotations, symbolic dynamics}
\renewcommand{\subjclassname}{MSC 2020}
\subjclass[2020]{Primary: 37A30, Secondary: 37B10, 47B36}

\begin{abstract}
We consider, for any irrational $\alpha$ and interval $I \subset \mathbb{T}$, the 2-interval coding subshift 
$X^{(I, \alpha)}$ induced by coding orbits under repeated rotation by $\alpha$ via membership in $I$ or $I^c$. Each sequence $c \in X^{(I, \alpha)}$ has an associated Schr\"{o}dinger operator $H_c$, and in \cite{kaminaga} it was proved that if the continued fraction of $\alpha$ has digits with limsup at least $4$, then almost every $c \in X^{(I, \alpha)}$ has so-called $3$-block Gordon structure, which implies that the operator $H_c$ has no eigenvalues. 

We significantly improve this result by completely characterizing almost-sure $3$-block Gordon structure, proving that in fact it holds for all $(\alpha,I)$ except for a countable set of pairs $(\alpha, |I|)$ where $\alpha$ is M\"{o}bius equivalent to the silver mean and $|I| \in \mathbb{Z}\alpha + f(\alpha)$ where $f(\alpha)$ is a specific infinite series
taking value either $\frac{1}{2}, \frac{\alpha}{2}$, or $\frac{\alpha+1}{2}$.

We also show that for a set of $\alpha$ of full measure, and for every $I$, the set of points whose orbit codings do not have $3$-block Gordon structure has Hausdorff dimension bounded away from $1$.
\end{abstract}

\maketitle

\section{Introduction}\label{sec:intro}

A general heuristic used in the study of Schr\"{o}dinger operators is that the less complex the underlying sequence, the smaller the spectrum of the associated operator. One of the most useful tools in proving such results is the $3$-block Gordon lemma (\cite{3block}). A sequence $x$ on a finite alphabet $A$ is said to have {\bf 3-block Gordon structure} if there exists a sequence $(q_k)$ such that for every $k$ and every $i \in (-q_k, 0]$, $x(i) = x(i - q_k) = x(i + q_k)$. The $3$-block Gordon lemma states that for $x$ with $3$-block Gordon structure, the operator $H_x$ has no eigenvalues. 

Outside of Schr\"{o}dinger operators, the study of such periodic structures has found multiple applications. For instance, in \cite{trans} and \cite{roy}, similar properties are applied to questions of transcendence and simultaneous Diophantine approximation.

From the viewpoint of symbolic dynamics, it is natural to study an entire subshift $X$ (closed shift-invariant set of sequences) and try to prove that most/all sequences in $X$ have desirable properties. Since it is not possible for all sequences of an infinite aperiodic subshift to have $3$-block Gordon structure (see \cite{no3block}), existing results generally prove that $3$-block Gordon structure holds for a set of positive measure for some natural ergodic measure(s). Then one can use ergodicity to show that almost every $x \in X$ has a shift with the structure. Such results have been proved for many natural `low-complexity' subshifts, such as Sturmian shifts (\cite{sturm}), some simple Toeplitz shifts (\cite{sellthesis}), and subshifts induced by primitive substitutions of index greater than $3$ (\cite{index3}).

A natural class for which only partial results exist are codings of a circle rotation by $\alpha$ by two intervals $I$ and $I^c$; we denote the subshift of such codings by $X^{(I,\alpha)}$. When $\alpha \in \mathbb{Q}$, all points in $X^{(I,\alpha)}$ trivially have $3$-block Gordon structure, so we for now on restrict to the case of irrational $\alpha$.
The best known result is proved in \cite{kaminaga}: if the rotation number $\alpha$ has continued fraction expansion with digits $a_k$ and $\limsup a_k \geq 4$, then $3$-block Gordon structure holds for a set with positive measure, and therefore absence of eigenvalues of the associated operator holds almost-surely. 
In particular, this demonstrated that for a set of $\alpha$ of full Lebesgue measure, almost-sure absence of eigenvalues holds for every $I$.

Our main result improves this significantly by completely characterizing such subshifts for which almost all $x \in X$ (up to a shift) have $3$-block Gordon structure. In fact, this holds for all but a very specific countable family of pairs
$(\alpha, |I|)$. Our proofs rely only on elementary results about continued fractions, primarily analysis of gaps in sets of the form $\{0, \alpha, 2\alpha, \ldots, n\alpha\}$ as treated in the Three-gap Theorem. (The statement relies on $f$ defined in Definition~\ref{fdef}.)

\begin{theorem}\label{mainthm}
The subshift $X^{(I, \alpha)}$ fails to have almost-sure $3$-block Gordon structure (up to a shift) if and only if  
$\lim a_n = 2$ and $|I| \in \mathbb{Z}\alpha + f(\alpha)$, where 
$f(\alpha) \in \{1/2, \alpha/2, (\alpha+1)/2\}$.
\end{theorem}
Combining this result with the 3-block Gordon Lemma, this yields the following simple corollary.
\begin{corollary}\label{Improving_eigenvalue_corollary}
    Let $\alpha$ be irrational and $I$ be a subinterval of the circle. If it is not the case that $\lim_{n \to \infty} a_n = 2$ and $2|I| \in \mathbb{Z}\alpha + \mathbb{Z}$, then for almost every $x \in X^{(I,\alpha)}$, the Schr\"{o}dinger operator $H_x$, has no eigenvalues.
\end{corollary}



Points in a subshift $X^{(I,\alpha)}$ are nearly in bijective correspondence with the unit circle. It is therefore possible to view the set of exceptional points (where $3$-block Gordon does not hold up to a shift) as a subset of the circle, which allows for a geometric analysis. We also show that for a set of $\alpha$ of full Lebesgue measure, 
the exceptional set of points in $[0,1)$ whose orbits do not have $3$-block Gordon structure actually has Hausdorff dimension less than $1$. (See \cite{DEGT} and \cite{JK} for other situations where zero measure is improved to Hausdorff dimension bounds for questions in the theory of ergodic Schr\"{o}dinger operators.)

\begin{theorem}\label{HD}
If $a_k \geq 4$ for a set of $k$ of positive upper density and $\limsup_k \left(\prod_{i=1}^k a_i\right)^{1/k} < \infty$, then $3$-block Gordon structure (up to a shift) holds for all $x \in X^{(I,\alpha)}$ except for codings of a set 
of Hausdorff dimension strictly less than $1$.
\end{theorem}

At least two very natural questions are raised by Theorem~\ref{mainthm} and \cref{Improving_eigenvalue_corollary}. Firstly, in the case where $\lim a_n = 2$ (meaning that $\alpha$ is M\"{o}bius equivalent to the silver mean $\sqrt{2} - 1$) and $|I| = m\alpha + f(\alpha)$, we do not know of an argument to prove or disprove the existence of eigenvalues; we suspect this case to be extremely interesting and give some thoughts in Section~\ref{silver}. However, the case remains open and we pose the following question.

\begin{question}
Do the two-interval circle rotation codings which fail to have almost-sure $3$-block Gordon structure (up to a shift) still have associated Schr\"{o}dinger operator with almost-sure absence of eigenvalues?
\end{question}

In (\cite{sturm}), the authors show that all Sturmian subshifts have uniform absence of eigenvalues. Our results show almost sure absence of eigenvalues in the general case of 2-Interval codings of circle rotations. Uniform absence in the general case is still open.

\begin{question}
Do two-interval circle rotation codings actually have associated Schr\"{o}dinger operators with uniform absence of eigenvalues?
\end{question}

\section{Definitions and preliminaries}\label{sec:defs}
\begin{definition}
For a starting value $x$, an interval $I$, and an irrational rotation number $\alpha$, define the associated \textbf{coding sequence} $c^{(x,I,\alpha)}$ by
\[
c^{(x,I,\alpha)}(n) = \chi_I(x + n\alpha)
\]
and the associated \textbf{coding subshift} $X^{(I,\alpha)} = \overline{\{c^{(x,I,\alpha)}\}_{x \in \mathbb{T}}}$.
\end{definition}
Informally, the coding sequence tracks whether or not the $n$th rotation of $x$ by $\alpha$ lies in $I$, and since the set $\{n\alpha\}$ is dense, the coding subshift contains all coding sequences regardless of starting value. As an extension of a group rotation which is $1$-$1$ on a set of full Haar measure, every coding subshift is \textbf{uniquely ergodic}, meaning that it has a unique shift-invariant measure.

Our tool for proving absence of eigenvalues for Schr\"{o}dinger operators of coding sequences is the classical $3$-block Gordon lemma, which precludes eigenvalues when the sequence exhibits certain periodic behavior at infinitely many scales.
\begin{definition}
For a finite alphabet $A$ and $x \in A^{\mathbb{Z}}$, $x$ has \textbf{3-block Gordon structure} if there exists a sequence $(q_k)$ such that for every $k$ and every $i \in (-q_k, 0]$, $x(i) = x(i - q_k) = x(i + q_k)$.
\end{definition}

The following lemma is classical (for instance, see Section 7.8 of \cite{schrobook}).

\begin{lemma}\label{gordon}
If $x \in A^\mathbb{Z}$ has $3$-block Gordon structure, then the associated Schr\"{o}dinger operator $H_x$ has no eigenvalues.
\end{lemma}

It is well-known that the behavior of coding sequences is largely controlled by the continued fraction expansion $\alpha = [a_1, a_2, a_3, \ldots]$ of $\alpha$; the larger the digits $a_n$, the more repetition the coding sequences exhibit. Although it is not possible for infinite subshifts to have uniform $3$-block Gordon structure (see \cite{no3block}), it has been shown that large enough digits imply almost-sure $3$-block Gordon structure for the associated coding subshift.

\begin{theorem}[\cite{kaminaga}]\label{kaminaga}
If $\limsup a_n \geq 4$, then almost all $x \in X^{(I, \alpha)}$ have a shift with $3$-block Gordon structure, and therefore $X^{(I, \alpha)}$ has almost-sure absence of eigenvalues.
\end{theorem}

The proof from \cite{kaminaga} basically relies on the observation that the distance from $q_k \alpha$ to the nearest integer is much smaller than all gaps between nearest elements of $\{0, \alpha, \ldots, (q_k-1)\alpha\}$ when $a_{k+1}$ is large, which leads to a high probability that $x(i) = x(i \pm q_k)$ for a randomly chosen coding sequence
$x \in X^{(I, \alpha)}$.

By using a more nuanced analysis of the structure of these gaps, we are able to prove $3$-block Gordon structure in a much greater generality. For simplicity, we reduce information about $I$ to its length $a$ from now on, 
as all intervals of the same length yield the same $X^{(I, \alpha)}$ up to sets of measure $0$.

We recall some elementary definitions and results from the theory of continued fractions; for a general introduction, see \cite{hardywright}.

\begin{definition}
For any irrational $\alpha$, one can write a continued fraction expansion of $\alpha$ as
\[
\alpha = a_0 + \cfrac{1}{a_1 + \cfrac{1}{a_2 + \cfrac{1}{a_3 + \ddots}}}
\]
This means specifically that the sequence of truncations 
\[
a_0 + \cfrac{1}{a_1 + \cfrac{1}{a_2 + \cfrac{1}{\ddots + \cfrac{1}{a_k}}}}
\]
approaches $\alpha$. This is called the sequence of \textbf{continued fraction convergents} to $\alpha$, and is written in lowest terms as $(p_k/q_k)$.
\end{definition}

\begin{definition}
For any $\alpha$ and $a$ and any $k \in \mathbb{N}$, define 
\[
Q_k = \{i\alpha \ : \ 0 \leq i < q_k\}, 
\]
where $(p_k/q_k)$ are the continued fraction convergents to $\alpha$, and define $\epsilon_k = d(0, q_k \alpha)$
(according to the standard metric on the $1$-torus $\mathbb{T}$).
We say that $q_k \alpha$ is \textbf{positive} if it lies in $(0, 1/2)$ (with the usual identification of $\mathbb{T}$ with $[0,1)$) and \textbf{negative} otherwise. 
\end{definition}

The signs of $q_k \alpha$ alternate as $k$ increases. The following recursion for continued fraction convergents is standard:
\[
q_{k+1} = a_{k+1} q_k + q_{k-1}.
\]

Multiplying by $\alpha$ and recalling that $q_{k} \alpha$ has opposite sign from $q_{k-1} \alpha$ and $q_{k+1} \alpha$ yields
\[
q_{k+1} \alpha = a_{k+1} q_k \alpha + q_{k-1} \alpha \Longrightarrow \epsilon_{k-1} = a_{k+1} \epsilon_k + \epsilon_{k+1}.
\]

The following lemma is now an exercise.

\begin{lemma}\label{ratios}
For any $k$, $\frac{q_{k}}{q_{k-1}} \in (a_k + (a_{k-1}+1)^{-1}, a_k + 1)$ and $\frac{\epsilon_{k-2}}{\epsilon_{k-1}} \in 
(a_{k} + (a_{k+1} + 1)^{-1}, a_{k} + 1)$.
\end{lemma}

One of the most celebrated results in Diophantine approximation is the Three-Gap Theorem, which states that for any irrational $\alpha$ and $n > 0$, when the numbers $0, \alpha, \ldots, n\alpha$ are placed on a unit circle, the $n$ intervals created (the `gaps') have at most three distinct lengths. We are interested specifically in the case where $n = q_k$ for some $k$, which actually implies only two different gaps; this implicitly appears in some original proofs of the Three-Gap Theorem, but we reference the more modern treatment \cite{3gap}.

\begin{lemma}[\cite{3gap}, Lemma 2.1]\label{longshort}
For any $\alpha$ and $k$, the set $Q_k = \{i\alpha \ : \ 0 \leq i < q_k\}$ has two different gaps, of lengths $S_k := \epsilon_{k-1}$ (short) and $L_k := \epsilon_{k-1} + \epsilon_k$ (long).
\end{lemma}

This immediately yields the following (also well-known) fact.

\begin{lemma}\label{qeps}
For any $\alpha$ and $k$, $q_k \epsilon_{k-1} < 1$ and $q_k \epsilon_k > (a_{k+1} + 2)^{-1}$. 
\end{lemma}

\begin{proof}
The first fact follows since the sum of the $q_k$ gaps in $Q_k$ is equal to $1$, and the gaps have lengths $S_k = \epsilon_{k-1}$ and $L_k > S_k$. The same logic immediately yields $q_k L_k > 1$, and so by \Cref{ratios},
\[
1 < q_k L_k = q_k (\epsilon_{k-1} + \epsilon_k) < q_k \epsilon_k (a_{k+1} + 2).
\]
\end{proof}

Finally, we need a simple description of which elements of $Q_k$ correspond to short/long intervals.

\begin{lemma}\label{endpts}
If $q_k \alpha$ is positive, then $i\alpha$ is the left endpoint of a long gap in $Q_k$ iff $0 \leq i < q_{k-1}$. If $q_k \alpha$ is negative, then this characterizes right endpoints of long gaps. 
\end{lemma}

\begin{proof}
When $q_k \alpha$ is positive, simply note that $i\alpha$ is the left endpoint of a short gap in $Q_k$ iff $i \alpha + S_k = (i - q_{k-1}) \alpha$ is in $Q_k$, which happens iff $i \geq q_{k-1}$. The argument for negative $q_k \alpha$ is trivially similar.
\end{proof}

Much of our proof relies on how intervals at each step break down at future steps; the following is easily checked.

\begin{theorem}\label{decomp}
If $q_k \alpha$ is positive, then each short interval in $Q_k$ breaks into, from left to right, $a_{k+1} - 1$ short intervals in $Q_{k+1}$ followed by a single long interval in $Q_{k+1}$, and each long interval in $Q_K$ breaks into, from left to right, $a_{k+1}$ short intervals in $Q_{k+1}$ followed by a single long interval in $Q_{k+1}$. If $q_k \alpha$ is negative, all descriptions above are instead from right to left. 
\end{theorem}

See Figure~\ref{fig1} for an illustration of how intervals break down over several steps according to Theorem~\ref{decomp}. 

\begin{figure}
\centering
\begin{tikzpicture}
\draw[thick] (0,0) -- (15, 0);
\filldraw[black] (0,0) circle (3pt);
\filldraw[black] (15,0) circle (3pt);
\node at (-0.5, 0){$0$};
\node at (15.5,0){$1$};
\node at (7.5,0.5){L};
\draw[thick] (0,-1.5) -- (15, -1.5);
\filldraw[black] (0,-1.5) circle (3pt);
\filldraw[black] (15,-1.5) circle (3pt);
\filldraw[black] (165/42,-1.5) circle (3pt);
\filldraw[black] (330/42,-1.5) circle (3pt);
\node at (82.5/42,-1){S};
\node at (247.5/42,-1){S};
\node at (480/42,-1){L};
\draw[thick] (0,-3) -- (15, -3);
\filldraw[black] (0,-3) circle (3pt);
\filldraw[black] (15,-3) circle (3pt);
\filldraw[black] (165/42,-3) circle (3pt);
\filldraw[black] (330/42,-3) circle (3pt);
\filldraw[black] (495/42,-3) circle (3pt);
\node at (82.5/42,-2.5){L};
\node at (247.5/42,-2.5){L};
\node at (412.5/42,-2.5){L};
\node at (562.5/42,-2.5){S};
\draw[thick] (0,-4.5) -- (15, -4.5);
\filldraw[black] (0,-4.5) circle (3pt);
\filldraw[black] (15,-4.5) circle (3pt);
\filldraw[black] (30/42,-4.5) circle (3pt);
\filldraw[black] (60/42,-4.5) circle (3pt);
\filldraw[black] (90/42,-4.5) circle (3pt);
\filldraw[black] (120/42,-4.5) circle (3pt);
\filldraw[black] (165/42,-4.5) circle (3pt);
\filldraw[black] (195/42,-4.5) circle (3pt);
\filldraw[black] (225/42,-4.5) circle (3pt);
\filldraw[black] (255/42,-4.5) circle (3pt);
\filldraw[black] (285/42,-4.5) circle (3pt);
\filldraw[black] (330/42,-4.5) circle (3pt);
\filldraw[black] (360/42,-4.5) circle (3pt);
\filldraw[black] (390/42,-4.5) circle (3pt);
\filldraw[black] (420/42,-4.5) circle (3pt);
\filldraw[black] (450/42,-4.5) circle (3pt);
\filldraw[black] (495/42,-4.5) circle (3pt);
\filldraw[black] (525/42,-4.5) circle (3pt);
\filldraw[black] (555/42,-4.5) circle (3pt);
\filldraw[black] (585/42,-4.5) circle (3pt);
\node at (15/42,-4){S};
\node at (45/42,-4){S};
\node at (75/42,-4){S};
\node at (105/42,-4){S};
\node at (142.5/42,-4){L};
\node at (180/42,-4){S};
\node at (210/42,-4){S};
\node at (240/42,-4){S};
\node at (270/42,-4){S};
\node at (307.5/42,-4){L};
\node at (345/42,-4){S};
\node at (375/42,-4){S};
\node at (405/42,-4){S};
\node at (435/42,-4){S};
\node at (472.5/42,-4){L};
\node at (510/42,-4){S};
\node at (540/42,-4){S};
\node at (570/42,-4){S};
\node at (607.5/42,-4){L};
\end{tikzpicture}
\caption{Intervals in $Q_k$ for $\alpha = 11/42 = [3,1,4,2]$ and $0 \leq k \leq 3$}\label{fig1}
\end{figure}
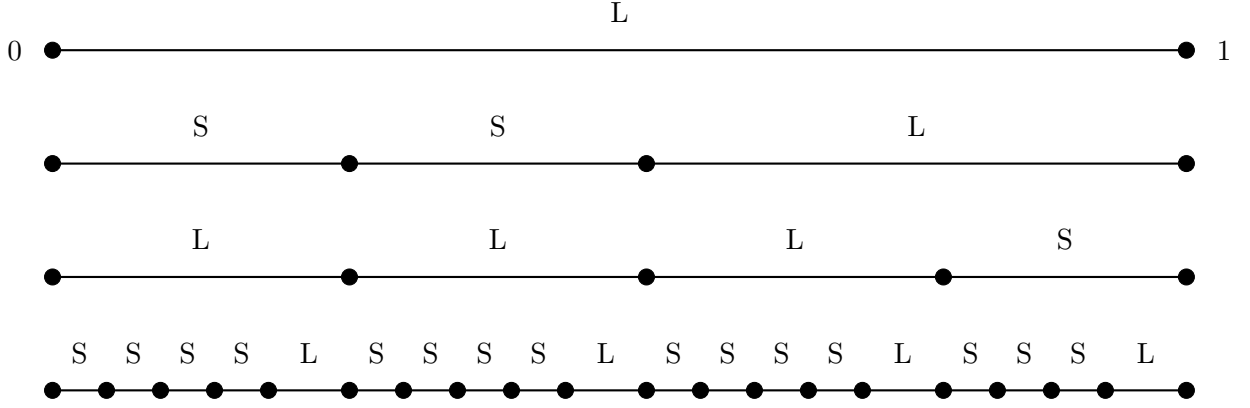

\section{Proofs}\label{sec:proofs}

The main idea of our proofs is to characterize whether a coding sequence has repetition around the origin in terms of location of the point of the circle being coded. 
Specifically, for a coding sequence $c = c^{(x,I,\alpha)}$, $c(i \pm q_k) = c(i)$ iff $x + i\alpha$ and 
$x + (i \pm q_k) \alpha$ lie in the same interval $I$ or $I^c$. We will from now on take $I$ to have endpoints $0$ and $a$ without loss of generality, since it does not change the space $X^{(I,\alpha)}$, and we also ignore coding sequences for $x \in \mathbb{Z}\alpha \cup (a + \mathbb{Z}\alpha)$, since they are a set of measure zero. We now make the following definition.

\begin{definition}
For any $\alpha$ and $a$ and any $k \in \mathbb{N}$, define $T_k = Q_k \cup (Q_k + a)$ and
$\beta_k = 1 - m(T_k + (-\epsilon_k, \epsilon_k))$.
\end{definition}

We can now make the following characterization.

\begin{lemma}\label{Tk}
If $\epsilon_k < \min(a, 1-a)$, then a coding sequence $c = c^{(x,I,\alpha)}$ satisfies $c(i \pm q_k) = c(i)$ for all $-q_k < i \leq 0$ iff $x \notin T_k + (-\epsilon_k, \epsilon_k)$.    
\end{lemma}

\begin{proof}
Essentially this was given above. Consider such a coding sequence. By definition, $x \in T_k + (-\epsilon_k, \epsilon_k)$. iff there exists $0 \leq i < q_k$ so that $x$ is within distance $\epsilon_k$ of either $i\alpha$ or $i\alpha + a$. In both cases, since $x - i\alpha$ is within distance $\epsilon_k$ of an endpoint of $I$ and since $\epsilon_k$ is less than the lengths of both $I$ and $I^c$, either $x - i\alpha - q_k \alpha$ or $x - i\alpha + q_k \alpha$ is in a different interval to $x - i\alpha$. Equivalently, $c(i), c(i \pm q_k)$ are not all equal.

The converse direction is obvious, given the observation that if $c(i), c(i + q_k), c(i - q_k)$ are not all equal, then either $c(i) \neq c(i + q_k)$ or $c(i) \neq c(i - q_k)$.

\end{proof}

This allows for a simple sufficient condition for $3$-block Gordon structure.

\begin{proposition}\label{longint}
For any $\alpha$ and $a$, if $\limsup \beta_k > 0$, then almost every $x \in X^{(I, \alpha)}$ has a shift with $3$-block Gordon structure.
\end{proposition}

\begin{proof}

The map $x \mapsto c^{(x,I,\alpha)}$ sending $x \in \mathbb{T}$ to its coding sequence is a measure-preserving bijection off of the countable set where it is not injective, and so it suffices to show that almost-every $x \in \mathbb{T}$ yields coding sequence with $3$-block Gordon structure.

Define $G_k$ to be the set of $x \in \mathbb{T}$ whose coding sequence has the $3$-block Gordon structure at scale $q_k$, i.e., $c^{(x,I,\alpha)}(i) = c^{(x,I,\alpha)}(i \pm q_k)$ for all $-q_k < i \leq 0$. For sufficently large $k$, $\epsilon_k < \min(a, 1-a)$, and for all such $k$, by Lemma~\ref{Tk} we have
\[
G_k^c = T_k + (-\epsilon_k, \epsilon_k).
\]
This means that $\mu(G_k) = \beta_k$. Then $m\left(\limsup G_k\right) \geq \limsup m(G_k) = \limsup \beta_k > 0$. This means that there is a subset of $X^{(I, \alpha)}$ of positive measure of points with $3$-block Gordon structure. The set of points with a shift with $3$-block Gordon structure is then shift-invariant and positive measure, therefore measure $1$ by ergodicity, completing the proof.

\end{proof}

For future reference, we note that $\beta_k$ can be interpreted in the following way: take all gaps between elements of $T_k$, i.e. the intervals which comprise its complement. Then $\beta_k$ is the sum, over all such gaps of length greater than $2\epsilon_k$, of the length minus $2\epsilon_k$. Informally, it measures the total overage of gaps' lengths beyond $2\epsilon_k$.

We are now prepared to prove \Cref{mainthm}, which will be broken into several results. We first describe the proof of the previous best known result, \Cref{kaminaga}, using our terminology and \Cref{longint}.

\begin{proof}[Proof of \Cref{kaminaga}]
Note that since $T_k + (-\epsilon_k, \epsilon_k)$ is a union of $2q_k$ intervals of length $2\epsilon_k$, 
$\beta_k > 1 - 4\epsilon_k q_k$. Since $\limsup a_k \geq 4$, there either exist infinitely many $k$ for which $a_{k+1} \geq 5$ or there exist infinitely many $k$ for which $a_{k+1} = 4$ and $a_{k+2} \leq 4$. In either case, $\epsilon_{k-1}/\epsilon_k \geq 4.2$ by Lemma~\ref{ratios}. Therefore, by Lemma~\ref{qeps}, for all such $k$,
\[
\beta_k > 1 - 4\epsilon_k q_k > 1 - 4/4.2 \epsilon_{k-1} q_k > 1/21.
\]
Since this holds for infinitely many $k$, we are finished by Proposition~\ref{longint}.
\end{proof}

A naive extension of this idea will not allow $\limsup a_k$ to be reduced to $3$, since $4 \epsilon_k q_k$ can absolutely be greater than $1$ in this case. However, we can use more nuanced information about the gap lengths from \Cref{longshort}. 

\begin{theorem}\label{case1}
If $\limsup a_k = 3$, then almost every point of $X^{(I, \alpha)}$ has a shift with $3$-block Gordon structure.
\end{theorem}

\begin{proof}
Since $\limsup a_k = 3$, there exist infinitely many $k$ for which $a_{k+1} = 3$ and $a_{k+2} \leq 3$. For any such $k$, by 
Lemma~\ref{ratios}, $S_k = \epsilon_{k-1} > 3.25 \epsilon_k$ and $L_k = S_k + \epsilon_k > 4.25 \epsilon_k$. Now, regardless of the value of $a$, we will show that there are so few elements of $Q_k + a$ that $T_k = Q_k \cup (Q_k + a)$ will still have many large gaps.

Denote by $N$ the number of long gaps in $Q_k$ containing two points in $Q_k + a$, and by 
$M$ the number of long intervals in $Q_k$ containing exactly one point in $Q_k + a$. 
Then $M + N \leq q_{k-1}$ since there are $q_{k-1}$ long gaps in $Q_k$, and of the $q_{k-1}$ 
gaps of length $L_k > 4.25 \epsilon_k$ and $q_k - q_{k-1}$ gaps of length $S_k > 3.25 \epsilon_k$ in $Q_k$, we have

\begin{itemize}
\item $q_{k-1} - M - N$ long intervals disjoint from $Q_k + a$
\item $M$ long intervals containing only one point of $Q_k + a$
\item at least $(q_k - q_{k-1}) - (q_k - M - 2N) = M + 2N - q_{k-1}$ short gaps disjoint from $Q_k + a$.
\end{itemize}

This implies that $T_k = Q_k \cup (Q_k + a)$ contains at least $q_{k-1} - M - N$ gaps of length $L_k$, 
at least $M$ gaps of length $L_k/2$, and at least $M + 2N - q_{k-1}$ gaps of length $S_k = L_k - \epsilon_k$. Therefore,
\begin{multline*}
\beta_k = 1 - m(T_k + (-\epsilon_k, \epsilon_k)) > (q_{k-1} - M - N) L_k + ML_k/2 + 
(M + 2N - q_{k-1})(L_k - \epsilon_k) \\
= q_{k-1} \epsilon_k + (M/2 + N)(L_k - 2\epsilon_k) > q_{k-1} \epsilon_k > q_k \epsilon_k / 4 > \frac{1}{20}
\end{multline*}
by Lemma~\ref{ratios}.

Since this holds for infinitely many values of $k$, $\limsup \beta_k \geq \frac{1}{20}$, and we are finished
by Proposition~\ref{longint}.

\end{proof}

Although we do not deal with it in detail in this work, we note that one can generalize the definition of
$X^{(I,\alpha)}$ to codings by partitions by $k > 2$ intervals. The proof of \cref{case1} is easily adapted
(by considering not just the sets $Q_k$ and $a + Q_k$ but shifts corresponding to each endpoint of the partition) to show that $\limsup a_n \geq 2k-1$ implies almost-sure $3$-block Gordon structure for every $k$-interval partition. This leads to the following natural direction for future work.



\begin{question}
What sufficient and/or necessary conditions can be given for almost-sure $3$-block Gordon structure of a $k$-interval circle coding subshift? In particular, for $k > 2$, is it still the case that almost-sure $3$-block Gordon structure holds for a `typical' rotation and partition?
\end{question}

We can now assume for remaining cases that eventually $a_k \leq 2$. To deal with this, we will need to not only use information about the gaps in $Q_k$ but also the locations of points in $Q_k + a$ within those gaps. (We remark that the proof of \Cref{case1} in fact would work not only for $T_k$, but for the union of $Q_k$ and ANY set with cardinality $q_k$).
The first step is to represent $a$ via nearby elements of $Q_k$. Specifically, for any $k$, we write $a = i_k \alpha + \eta_k$ for $0 \leq i_k < q_k$ and $\eta_k$ with the same sign as $q_k \alpha$ and with minimal absolute value. 

Then, we can write
\[
Q_k + a = \{i\alpha + (i_k \alpha + \eta_k) \ : \ 0 \leq i < q_k\} = 
\{j \alpha + \eta_k \ : \ i_k \leq i < q_k\} \cup \{j \alpha + (q_k \alpha + \eta_k) \ : \ 0 \leq j < i_k\}.
\]
In other words, we can think of $Q_k + a$ as coming from shifting the first $i_k$ elements of $Q_k$ by $\eta_k + \epsilon_k$ and the rest by $\eta_k$ (in direction dependent on sign of $q_k \alpha$). This interpretation will allow us to conclude that $T_k$ has large enough intervals in many cases. 

We first prove a very useful auxiliary fact, showing that values of $k$ where $i_k$ and $\eta_k$ are simultaneously small lead to values of $\beta_k$ bounded away from $0$.

\begin{lemma}\label{smallboth}
For any $k$, if $a_k, a_{k+1}, a_{k+2} \leq 2$, $i_k/q_k < 0.25$ and $\eta_k/\epsilon_k < 0.25$, then $\beta_k > 0.001$.
\end{lemma}

\begin{proof}
Assume that for some value of $k$, $i_k/q_k < 0.25$ and $\eta_k/\epsilon_k < 0.25$. Recall that we can write
$Q_k + a = U_k \cup V_k$, where 
\begin{equation}\label{uv}
U_k = \{j \alpha + \eta_k \ : \ i_k \leq j < q_k\} \textrm{ and } V_k = \{j \alpha + (q_k \alpha + \eta_k) \ : \ 0 \leq j < i_k\}.
\end{equation}
Assume without loss of generality that $q_k \alpha$ is positive; then $\eta_k > 0$ and $q_k \alpha = \epsilon_k$. 
$U_k$ consists of a very large number of elements of $Q_k$ shifted by $\eta_k$ to the right, and $V_k$ consists of a very small number of elements of $Q_k$ shifted by $\eta_k + \epsilon_k$ to the right. We note that both of these numbers are smaller than 
$1.25\epsilon_k$, which is smaller than the length $S_k = \epsilon_{k-1}$ of a short gap in $Q_k$ by Lemma~\ref{ratios}. So, each gap in $T_k$ is obtained by splitting a gap in $Q_k$ by the unique point in $Q_k + a$ lying in it. 

Also note that the length of a long gap in $Q_k$ is $L_k = \epsilon_k + \epsilon_{k-1}$. By Lemma~\ref{ratios}, 
$\epsilon_{k-1} > \frac{4}{3} \epsilon_k$, so $L_k > (7/3) \epsilon_k$. Therefore, every long gap split by a point of $U_k$ leaves a gap of length $L_k - \eta_k > (25/12) \epsilon_k$. 

Every $j < q_{k-1}$ for which $i_k + j < q_{k-1}$ corresponds to a point in $U_k$ which is the shift by $\eta_k$ of a left endpoint of a long gap, and so there are at least $q_{k-1} - i_k$ gaps of length at least $(25/12) \epsilon_k$ in $T_k$. This means that $\beta_k > (q_{k-1} - i_k) \epsilon_k/12$.

Finally, note that $q_{k-1} > q_k/3$ by Lemma~\ref{ratios}, so since $i_k < q_k/4$, we see that
$q_{k-1} - i_k > q_k/12$. 
Therefore, by \Cref{qeps},
\[
\beta_k > q_k \epsilon_k/144 > 1/576 > 0.001.
\]

\end{proof}

We now will show that extremely small $i_k$ or extremely small distance from $Q_k$ alone yield large enough $\beta_{k'}$ for some nearby $k'$.

\begin{lemma}\label{smalli}
For any $k$, if $a_{k + j} \leq 2$ for $-8 \leq j \leq 2$ and $i_k/q_k < \frac{1}{10000}$, then $\beta_{k-8} > 0.001$.
\end{lemma}

\begin{proof}

Assume without loss of generality that $q_k \alpha$ is positive, i.e. equal to $\epsilon_k$. Then, we note that
$q_k = a_k q_{k-1} + q_{k-2} = (a_k a_{k-1} + 1) q_{k-2} + a_{k-2} q_{k-3} < 7 q_{k-2}$, and similarly
$q_{k-2} < 7q_{k-4} < 49q_{k-6} < 343q_{k-8}$, so $i_k < q_k/9604 < q_{k-8}/4$. Therefore, since all gaps at level $k-8$ are greater than those at level $k$ (and since $q_{k-8} \alpha$ is positive), $i_{k-8} = i_k$ and $\eta_{k-8} = \eta_k$. Then
\[
\eta_{k-8}/\epsilon_{k-8} = \frac{\eta_k}{\epsilon_{k-8}} < (\epsilon_{k-1} + \epsilon_k)/\epsilon_{k-8}.
\] 
By Lemma~\cref{ratios}, $\epsilon_{i-1} > 4/3 \epsilon_i$ for $k-7 \leq i \leq k$. Therefore, 
\[
\eta_{k-8}/\epsilon_{k-8} < (\epsilon_{k-1} + \epsilon_k)/\epsilon_{k-8} < (3/4)^7 + (3/4)^8 < 1/4.
\]
Since we already showed that $i_{k-8} = i_k < q_{k-8}/4$, we now are finished by Lemma~\ref{smallboth}.

\end{proof}

\begin{lemma}\label{smalldist}
For any $k$, if $a_{k + j} \leq 2$ for $1 \leq j \leq 9$ and $d(a, Q_k) < \frac{\epsilon_k}{10000}$, then 
$\max(\beta_{k+6}, \beta_{k+7}) > 0.001$.
\end{lemma}

\begin{proof}
By assumption, we can write $a = i \alpha + \eta$ for some $0 \leq i < q_k$ and $\eta$ with $|\eta| < \frac{\epsilon_k}{10000}$. We assume without loss of generality that $q_k \alpha$ is positive and break into cases depending on sign of $\eta$.\\

\noindent
\textbf{Case 1:} $\eta > 0$. Then $i_k = i$ and $\eta_k = \eta$. We now consider the value $k + 6$. 
By Lemma~\ref{ratios}, $\epsilon_{k+6} > \epsilon_k/729$, and so $\eta < \epsilon_k/10000 < \epsilon_{k+6}/4$. This means that 
$i_{k+6} = i$ and $\eta_{k+6} = \eta < \epsilon_{k+6}/4$. Finally, $q_{k+6} \geq (4/3)^6 q_k > 4q_k$, so $i \leq q_k < q_{k+6}/4$. and now we have $\beta_{k+6} > 0.001$ by Lemma~\ref{smallboth}.\\

\noindent
\textbf{Case 2:} $\eta < 0$. Then we look at $k+1$; $q_{k+1} \alpha$ is negative, and so $i_{k+1} = i$ and $\eta_{k+1} = \eta$. 
By Lemma~\ref{ratios}, $\epsilon_{k+7} > \epsilon_k/2187$, and so $\eta < \epsilon_k/10000 < \epsilon_{k+7}/4$. This means that
$i_{k+7} = i$ and $\eta_{k+7} = \eta < \epsilon_{k+7}/4$. Finally, $q_{k+7} \geq (4/3)^7 q_k > 4q_k$, so $i \leq q_k < q_{k+7}/4$ and so $\beta_{k+7} > 0.001$ by Lemma~\ref{smallboth}.

\end{proof}

This yields the following immediate corollary via Proposition~\ref{longint}. 

\begin{corollary}\label{bdaway}
If $\limsup a_k \leq 2$ and either $\liminf i_k/q_k < 0.0001$ or $\liminf d(a, Q_k)/\epsilon_k < 0.0001$, then almost all points of $X^{(I, \alpha)}$ have $3$-block Gordon structure.
\end{corollary}

We can now deal with the second case of \Cref{mainthm}.

\begin{theorem}\label{onetwo}
If $\limsup a_k = 2$ and $\liminf a_k = 1$, then almost every $x \in X^{(\alpha, I)}$ has a shift with $3$-block Gordon structure.
\end{theorem}

\begin{proof}
By the hypotheses of the theorem, there exists $K$ so that for $k > K$, $1 \leq a_k \leq 2$, and there are infinitely many values of $k$ for which $a_k = 1$ and for which $a_k = 2$. We can also assume (by increasing $K$ if necessary) that
$i_k/q_k, d(a, Q_k)/\epsilon_k \geq 0.0001$ for all $k > K$, since if this is not the case then the proof is complete by 
Proposition~\ref{longint} and Corollary~\ref{bdaway}. We assume for now that $q_k \alpha = \epsilon_k$ is positive without loss of generality.

Our proof relies on breaking into a variety of cases, and showing that each yields some $\beta_n$ bounded away from $0$, and then showing that this must happen infinitely many times. These cases will depend on whether $a$ is part of long or short intervals of length $S_k$ at various values of $k$. We remind the reader that by Lemma~\ref{endpts},
\begin{align*}
i \alpha \textrm{ is the left endpoint of a long interval in } Q_k & \Longleftrightarrow 0 \leq i < q_{k-1}\\
i \alpha \textrm{ is the right endpoint of a long interval in } Q_k & \Longleftrightarrow q_k - q_{k-1} \leq i < q_{k}.
\end{align*}
We refer to a value of $k$ when $a_{k+1} = 2$ as a $2$-split (since $\epsilon_{k-1} = 2\epsilon_{k} + \epsilon_{k+1}$, meaning that a short interval in $Q_k$ splits into two intervals, one long and one short, in $Q_{k-1}$) and a value of $k$ when $a_{k+1} = 1$ as a $1$-split for analogous reasons.\\

\noindent
\textbf{Case 1:} $a$ is part of a long interval for some $k > K$ with $a_{k+1} = 2$ (long gap before a $2$-split)\\

Suppose that this is the case. Then $i_k < q_{k-1}$, and based on the decomposition $Q_k + a = U_k \cup V_k$ from (\ref{uv}),
we see that all points of the form $j\alpha$ for $0 \leq j < i_k$ are shifted to the right by $\eta_k + \epsilon_k$, and all points of the form $j\alpha$ for $i_k \leq j < q_{k}$ are shifted to the right by $\eta_k$. 

Let us first suppose that $\eta_k \in (S_k, L_k) = (\epsilon_{k-1}, \epsilon_{k-1} + \epsilon_k)$. Then, consider all points $j\alpha$ for $q_{k-1} \leq j < (4/3)q_{k-1}$. These are all left endpoints of a short interval to the left of a long interval, meaning that $j\alpha + \eta_k$ will lie inside a long interval, at distance less than $\epsilon_k$ from the left endpoint. This implies that after its partition by points of $Q_k + a$, that long interval will still have a gap of length at least $S_k = \epsilon_{k-1} > 7/3 \epsilon_k$. Since there are $q_{k-1}/3$ such intervals, we see that by \Cref{qeps},
\[
\beta_k > (q_{k-1}/3)(\epsilon_k/3) > q_k \epsilon_k/27 > 1/108.
\]

We now assume that $\eta_k < S_k$. Then since $i_k/q_k > 0.0001$, we know that at least $0.0001q_k$ left endpoints of long gaps are shifted to the right by $\eta_k + \epsilon_k$. If $\eta_k > (7/6) \epsilon_k$, then this yields $0.0001q_k$ long gaps which retain an interval of length at least $(13/6) \epsilon_k$, which implies by \Cref{qeps} that
\[
\beta_k > (q_k/10000)(\epsilon_k/6) > q_k \epsilon_k/60000 > 1/240000.
\]
Before treating the final case $\eta_k \leq (7/6) \epsilon_k$, we note that $i_k$ cannot be too close to $q_{k-1}$: we can write
\[
-a = -i_k \alpha - \eta_k = (q_{k-1} - i_k) \alpha + (-q_{k-1} \alpha - \eta_k) = (q_{k-1} - i_k)\alpha + (\epsilon_{k-1} - \eta_k),
\]
corresponding to a decomposition of $-a$ as $i'_k \alpha + \eta'_k$ for $i'_k = q_{k-1} - i_k$ and 
$\eta'_k = \epsilon_{k-1} - \eta_k$. Since $Q_k \cup (Q_k - a) = T_k - a$ has exactly the same gaps as $T_k$, we may assume without loss of generality that $i'_k/q_k > 0.0001$ just as we did for $i_k$. In other words, we know that $i_k = q_{k-1} - i'_k < q_{k-1} - 0.0001q_k$. 

Now we treat the final case $\eta_k \leq (7/6) \epsilon_k$. By the bound just proved, there are at least $0.0001 q_k$ long intervals whose left endpoint is shifted to the right by $\eta_k$ in $Q_k + a$. Every such interval then retains a gap of length at least $L_k - (7/6)\epsilon_k > (13/6) \epsilon_k$ in $T_k$, which again implies by \Cref{qeps} that
\[
\beta_k > (q_k/10000)(\epsilon_k/6) > q_k \epsilon_k/60000 > 1/240000.
\]
We have then shown that in Case 1, $\beta_k > 1/240000$. 
\begin{flushright}
    $\square$\\
\end{flushright}

\noindent
\textbf{Case 2:} $a$ is part of a short interval for some $k > K$ with $a_{k+1} = 2$ and a long interval for $k+1$ 
(short gap before a $2$-split and long gap after)\\

\begin{figure}[h]
\centering
\begin{tikzpicture}
\draw[thick] (0,0) -- (15, 0);
\filldraw[black] (0,0) circle (3pt);
\filldraw[black] (15,0) circle (3pt);
\filldraw[red] (9,0) circle (5pt);
\draw[<->] (0, -.25) -- (9, -.25);
\node at (9,-0.5){a};
\node at (4.5,-0.5){$\eta_k$};
\node at (7.5,0.5){S};
\draw[thick] (0,-1.5) -- (15, -1.5);
\filldraw[black] (0,-1.5) circle (3pt);
\filldraw[black] (15,-1.5) circle (3pt);
\filldraw[black] (6,-1.5) circle (3pt);
\filldraw[red] (9,-1.5) circle (5pt);
\draw[<->] (0, -1.75) -- (6, -1.75);
\node at (3,-2){$\epsilon_k$};
\node at (9,-2){a};
\node at (3,-1){S};
\node at (10.5,-1){L};
\end{tikzpicture}
\caption{Case 2}\label{fig2}
\end{figure}

Suppose that this is the case. Since $a_{k+1} = 2$, the short interval $a$ lies in at step $k$ becomes a short interval to the left of a long interval in step $k+1$, and since $a$ is in the long interval there, $\eta_k$ must be greater than the length of a short interval at step $k+1$, which is $\epsilon_k$. In fact, then $\eta_k > \epsilon_k + 0.0001\epsilon_{k+1}$, by our assumption that $d(a, Q_{k+1})/\epsilon_{k+1} \geq 0.0001$. 

In addition, by Lemma~\ref{endpts}, $i_k \geq q_{k-1}$. This means that all points of the form $j\alpha + \eta_k + \epsilon_k$ for $0 \leq j < q_{k-1}$ are in $T_k$, meaning in turn that all long intervals in $Q_k$ are split by a point at least $\eta_k + \epsilon_k > 2\epsilon_k + 0.0001\epsilon_{k+1}$ from the left endpoint. Then by \Cref{qeps}, in Case 2,
\[
\beta_k > q_{k-1}(\epsilon_{k+1}/10000) > q_{k+1} \epsilon_{k+1}/90000 > 1/360000.
\]
\begin{flushright}
    $\square$\\
\end{flushright}

By Cases 1 and 2, we can assume WLOG that for $k > K$, $a$ is in a short gap before AND after every $2$-split.
(Otherwise the proof would be complete by \Cref{longint}.)
Consider any $k$ for which $a_{k+1} = 2$ and $a_{k+2} = 1$ (there are infinitely many of these). Then we know that $a$ was in a short interval at levels $k$ and $k+1$. The short interval from level $k+1$ becomes a single long interval at level $k+2$ since $k+1$ is a $1$-split. This means that $k+2$ cannot be a $2$-split, else we are in Case 1. Therefore, 
$a_{k+3} = 1$. \\

\noindent
\textbf{Case 3:} $a$ is part of a short interval at step $k+3$, $a_{k+1} = 2$, $a_{k+2} = a_{k+3} = 1$ 
(short, then short, then long, then short, for a $2$-split followed by two $1$-splits.)\\

\begin{figure}[h]
\centering
\begin{tikzpicture}
\draw[thick] (0,0) -- (15, 0);
\filldraw[black] (0,0) circle (3pt);
\filldraw[black] (15,0) circle (3pt);
\filldraw[red] (1.5,0) circle (5pt);
\draw[<->] (0, -.25) -- (1.5, -.25);
\node at (1.5,-0.5){a};
\node at (0.75,-0.5){$\eta_k$};
\node at (7.5,0.5){S};
\draw[thick] (0,-1.5) -- (15, -1.5);
\filldraw[black] (0,-1.5) circle (3pt);
\filldraw[black] (15,-1.5) circle (3pt);
\filldraw[black] (45/8,-1.5) circle (3pt);
\filldraw[red] (1.5,-1.5) circle (5pt);
\node at (1.5,-2){a};
\node at (45/16,-1){S};
\node at (165/16,-1){L};
\draw[thick] (0,-3) -- (15, -3);
\filldraw[black] (0,-3) circle (3pt);
\filldraw[black] (15,-3) circle (3pt);
\filldraw[black] (45/8,-3) circle (3pt);
\filldraw[black] (45/4,-3) circle (3pt);
\filldraw[red] (1.5,-3) circle (5pt);
\node at (1.5,-3.5){a};
\node at (45/16,-2.5){L};
\node at (135/16,-2.5){L};
\node at (105/8,-2.5){S};
\draw[thick] (0,-4.5) -- (15, -4.5);
\filldraw[black] (0,-4.5) circle (3pt);
\filldraw[black] (15,-4.5) circle (3pt);
\filldraw[black] (7.5,-4.5) circle (3pt);
\filldraw[black] (15/8,-4.5) circle (3pt);
\filldraw[black] (45/8,-4.5) circle (3pt);
\filldraw[black] (45/4,-4.5) circle (3pt);
\filldraw[red] (1.5,-4.5) circle (5pt);
\draw[<->] (0, -4.75) -- (15/8, -4.75);
\node at (0.75,-5){$\epsilon_{k+2}$};
\node at (1.5,-5){a};
\node at (15/16,-4){S};
\node at (30/8,-4){L};
\node at (105/16,-4){S};
\node at (75/8,-4){L};
\node at (105/8,-4){L};
\end{tikzpicture}
\caption{Case 3}\label{fig3}
\end{figure}

The short interval $a$ lies in at step $k+3$ shares a left endpoint with the short interval from step $k$, and so $\eta_k$ is less than the length of a short interval at step $k+3$, which is $\epsilon_{k+2}$. Then as in Case 2, all left endpoints of long intervals at step $k$ get shifted to the right by $\eta_k + \epsilon_k$ in $Q_k + a$. This shift has length less than $\epsilon_k + \epsilon_{k+2}$. We note that by the standard recursions, the length of a long gap at level $k$ is
$\epsilon_{k-1} + \epsilon_k = 3\epsilon_k + \epsilon_{k+1} = 3\epsilon_k + \epsilon_{k+2} + \epsilon_{k+3}$. 

Therefore, all long gaps at level $k$ retain a gap of length $L_k - (\epsilon_k + \epsilon_{k+2}) = 2\epsilon_k + \epsilon_{k+3}$ in $T_k$. Therefore, by \Cref{qeps}, in Case 3 we have
\[
\beta_k > q_{k-1}\epsilon_{k+3} > q_k \epsilon_k/81 > 1/324.
\]
\begin{flushright}
    $\square$\\
\end{flushright}

The proof would be complete if Case 3 happened infinitely often, and so we can assume that for sufficiently large $k$, 
$a$ is part of a long interval at step $k+3$, implying that $a_{k+4} = 1$ (since we've assumed Case 1 does not occur for large $k$). We can then consider the position of $a$ starting from step $k+1$.\\

\noindent
\textbf{Case 4:} $a$ is part of a short interval at step $k+4$ and $a_{k+2} = a_{k+3} = a_{k+4} = 1$ 
(short, then long, then long, then short, for three $1$-splits)\\

\begin{figure}[h]
\centering
\begin{tikzpicture}
\draw[thick] (0,0) -- (15, 0);
\filldraw[black] (0,0) circle (3pt);
\filldraw[black] (15,0) circle (3pt);
\filldraw[red] (13,0) circle (5pt);
\node at (13,-0.5){a};
\node at (7.5,0.5){S};
\draw[thick] (0,-1.5) -- (15, -1.5);
\filldraw[black] (0,-1.5) circle (3pt);
\filldraw[black] (15,-1.5) circle (3pt);
\filldraw[red] (13,-1.5) circle (5pt);
\node at (13,-2){a};
\node at (7.5,-1){L};
\draw[thick] (0,-3) -- (15, -3);
\filldraw[black] (0,-3) circle (3pt);
\filldraw[black] (15,-3) circle (3pt);
\filldraw[black] (6,-3) circle (3pt);
\filldraw[red] (13,-3) circle (5pt);
\node at (13,-3.5){a};
\node at (3,-2.5){S};
\node at (10.5,-2.5){L};
\draw[thick] (0,-4.5) -- (15, -4.5);
\filldraw[black] (0,-4.5) circle (3pt);
\filldraw[black] (15,-4.5) circle (3pt);
\filldraw[black] (6,-4.5) circle (3pt);
\filldraw[black] (12,-4.5) circle (3pt);
\filldraw[red] (13,-4.5) circle (5pt);
\draw[<->] (12, -4.75) -- (15, -4.75);
\node at (13.5,-5){$\epsilon_{k+3}$};
\node at (3,-4){L};
\node at (9,-4){L};
\node at (13.5,-4){S};
\end{tikzpicture}
\caption{Case 4}\label{fig4}
\end{figure}

We will end up showing that $\beta_{k+2}$ is bounded away from $0$ in this case, but first need to point out a new subtlety that arises since $a_{k+2} = 1$ (i.e. since we are making an argument at a value of $k$ corresponding to a $1$-split). Previously, as long as we could show that enough long gaps in $Q_k$ contained a point of $Q_k + a$ close to an endpoint, we were finished, because even though a second point of $Q_k + a$ could theoretically lie in that long interval, $S_k = \epsilon_{k-1} > (7/3) \epsilon_k$ was bounded away from $2\epsilon_k$, and so we were guaranteed a long enough interval in $T_k$. Now, $S_k < 2\epsilon_k$, so we need to take more care with this issue. 

In the situation of Case 4, $a$ is within distance $S_{k+4} = \epsilon_{k+3}$ of the right endpoint of the long interval $a$ lies in at level $k+2$. This means that we can, at level $k+2$, write $a = i\alpha + \eta$ for $-\epsilon_{k+3} < \eta < 0$. (This is not the standard representation of $a$ at level $k+2$, since $q_{k+2} \alpha$ is positive and $\eta$ is negative. But we will use it for this case only.) In fact we can say more; by assumption, $a$ is distance at least $0.0001\epsilon_{k+4}$ from all points in 
$Q_{k+4}$, and so $\eta > 0.0001\epsilon_{k+4} - \epsilon_{k+3}$. 

Since $a$ is near an endpoint of a long interval at level $k+2$ which was already in $Q_{k+1}$, we know that $i < q_{k+1}$. Using a similar argument with negating $a$ as in Case 1, we see that
\[
-a = -i\alpha - \eta = (q_{k+1} - i)\alpha + (-q_{k+1}\alpha - \eta) = (q_{k+1} - i)\alpha + (\epsilon_{k+1} - \eta).
\]
Since we assumed that $-a$ does not have a representation with infinitely many $i_k/q_k < 0.0001$, as before we know 
that $i < q_{k+1} - 0.0001 q_{k+2}$. 

Now, consider any point in $Q_k$ of the form $j \alpha$ with $q_{k+2} - q_{k+1} < j < 1.0001q_{k+2} - q_{k+1}$. This is the right endpoint of a long interval, and after adding $a$, it is $(j + i)\alpha + \eta$, which is a right endpoint of a long interval shifted to the left by $|\eta| < \epsilon_{k+3} - 0.0001\epsilon_{k+4}$. 

This means that it retains an interval of length at least 
\[
L_{k+2} - (\epsilon_{k+3} - 0.0001\epsilon_{k+4}) = \epsilon_{k+1} + \epsilon_{k+2} - \epsilon_{k+3} + 0.0001\epsilon_{k+4}
= 2\epsilon_{k+2} + 0.0001\epsilon_{k+4}.
\]

Since there were $0.0001q_{k+2}$ such intervals, by \Cref{qeps},
\[
\beta_{k+2} > (\epsilon_{k+4}/10000)(q_{k+2}/10000) > \epsilon_{k+2}q_{k+2}/400000000 > 1/1600000000
\]
in Case 4.
\begin{flushright}
    $\square$\\
\end{flushright}

The proof would again be complete if Case 4 occurred infinitely often, so for large $k$ we may assume that $a$ lies in a long interval at levels $k+2$, $k+3$, and $k+4$, and $a_{k+1}, a_{k+2}, a_{k+3} = 1$. Since $\limsup a_k = 2$, there must exist some $k' > k+2$ which is a $2$-split, implying that $a$ lies in a short interval at level $k'$ by Case 1. We may then take $k'$ to be the minimal value greater than $k+2$ where $a$ lies in a short interval, meaning that $a_{k'-2}, a_{k'-1}, a_{k'} = 1$ (and $a_{k'+1} = 2$).\\

\noindent
\textbf{Case 5:} $a$ is part of a long interval at levels $k' - 3$, $k' - 2$, $k' - 1$, and a short interval at level $k$
(long, then long, then long, then short, for three $1$-splits).\\

\begin{figure}[h]
\centering
\begin{tikzpicture}
\draw[thick] (0,0) -- (15, 0);
\filldraw[black] (0,0) circle (3pt);
\filldraw[black] (15,0) circle (3pt);
\filldraw[red] (7,0) circle (5pt);
\node at (7,-0.5){a};
\node at (7.5,0.5){L};
\draw[<->] (0, -.25) -- (7, -.25);
\node at (3.5,-0.5){$\eta_{k'-3}$};
\draw[thick] (0,-1.5) -- (15, -1.5);
\filldraw[black] (0,-1.5) circle (3pt);
\filldraw[black] (15,-1.5) circle (3pt);
\filldraw[black] (75/13,-1.5) circle (3pt);
\filldraw[red] (7,-1.5) circle (5pt);
\draw[<->] (0, -1.75) -- (75/13, -1.75);
\node at (75/26,-2){$\epsilon_{k'-3}$};
\node at (7,-2){a};
\node at (75/26,-1){S};
\node at (135/13,-1){L};
\draw[thick] (0,-3) -- (15, -3);
\filldraw[black] (0,-3) circle (3pt);
\filldraw[black] (15,-3) circle (3pt);
\filldraw[black] (75/13,-3) circle (3pt);
\filldraw[black] (150/13,-3) circle (3pt);
\filldraw[red] (7,-3) circle (5pt);
\node at (7,-3.5){a};
\node at (75/26,-2.5){L};
\node at (225/26,-2.5){L};
\node at (172.5/13,-2.5){S};
\draw[thick] (0,-4.5) -- (15, -4.5);
\filldraw[black] (0,-4.5) circle (3pt);
\filldraw[black] (15,-4.5) circle (3pt);
\filldraw[black] (75/13,-4.5) circle (3pt);
\filldraw[black] (150/13,-4.5) circle (3pt);
\filldraw[black] (30/13,-4.5) circle (3pt);
\filldraw[black] (105/13,-4.5) circle (3pt);
\filldraw[red] (7,-4.5) circle (5pt);
\node at (7,-5){a};
\node at (15/13,-4){S};
\node at (105/26,-4){L};
\node at (90/13,-4){S};
\node at (255/26,-4){L};
\node at (172.5/13,-4){L};
\end{tikzpicture}
\caption{Case 5}\label{fig5}
\end{figure}

Since we know nothing about the parity of $k' - k$, here we cannot assume anything about the sign of $q_{k' - 3} \alpha$ without loss of generality. We still begin with the case where it is positive and then describe required changes in the opposite case.

Since $a$ lies in long intervals at levels $k'-3$, $k'-2$, $k'-1$, and short at $k'$, the short interval it lies in at level $k'$ shares left endpoint with the long interval from level $k'-2$. In addition, recall that $d(a, Q_{k'-2}) > 0.0001 \epsilon_{k'-2}$. Therefore, $\eta_{k'-3} \in (\epsilon_{k'-3} + 0.0001 \epsilon_{k'-2}, \epsilon_{k'-3} + \epsilon{k'-1})$. Since $a$ lies in a long gap at level $k'-3$, by Lemma~\ref{endpts} $i_{k'-3} < q_{k'-4}$. 


Now, consider any point in $Q_k$ of the form $j \alpha$ with $0.9999 q_{k'-3} < j < q_{k'-3}$. This is the right endpoint of a long interval, and after adding $a$, it is $(j + i_{k'-3} - q_{k'-3})\alpha + (\eta_{k'-3} + \epsilon_{k'-3})$, which is a left endpoint of a long interval shifted to the right by 
\[
\eta_{k'-3} + \epsilon_{k'-3} \in (2 \epsilon_{k'-3} + 0.0001\epsilon_{k'-2}, 2\epsilon_{k'-3} + \epsilon_{k'-1}).
\] 
This means that the corresponding long interval retains an interval of length at least $2\epsilon{k'-3} + 0.0001 \epsilon_{k'-2}$ in $T_k$. Since there are $0.0001 q_{k'-3}$ such intervals,
\[
\beta_{k'-3} > (\epsilon_{k'-2}/10000)(q_{k'-3}/10000) > \epsilon_{k'-3}q_{k'-3}/200000000 > 1/800000000.
\]

If instead $q_{k'-3} \alpha$ were negative, all left and right endpoints in the proof are swapped, and the sign of $\eta_{k'-3}$ will be negative, but nothing else needs to change. Therefore, in Case 5, $\beta_{k'-3} > 1/800000000$.
\begin{flushright}
    $\square$\\
\end{flushright}

Between any $k_1, k_2$ for which $a_{k_1} = a_{k_2} = 2$ and $a_{k_1 + 1} = a_{k_2 + 1} = 1$, one of Cases 1-5 must occur, and so there is a value of $k$ for which $\beta_k > 10^{-9}$. Since there are infinitely many such $k$, by Proposition~\ref{longint}, the proof is complete.

\end{proof}

The third and fourth cases of \Cref{mainthm} now follow from an analysis of the possibilities for membership of $a$ in long/short intervals not already treated in \Cref{onetwo}. We first need a simple observation.

\begin{lemma}\label{parity}
If $\lim a_k = 2$, then the parities of $p_k + p_{k+1}$ and $q_k + q_{k+1}$ are eventually constant and they cannot both be eventually even.
\end{lemma}

\begin{proof}
Since $a_k$ is eventually $2$, for sufficiently large $k$ we have 
\[
p_{k+1} = 2p_k + p_{k-1} \textrm{ and } q_{k+1} = 2q_k + q_{k-1}.
\]
Therefore, 
\[
q_{k+1} - q_k = q_k + q_{k-1} = (q_k - q_{k-1}) + 2q_{k-1} \textrm{ and } p_{k+1} - p_k = p_k + p_{k-1} = (p_k - p_{k-1}) + 2p_{k-1},
\]
proving claimed eventual constancy of these parities. If both were eventually even, then (recalling that $p_k/q_k$ is always in lowest terms), $p_k$ and $q_k$ would have to eventually be odd. Since $p_{-1} = 0$ and $q_{-1} = 1$, there must be a maximal value $m$ for which $p_m$ and $q_m$ have opposite parities. We now have a contradiction, since the recursions
\[
p_{m+2} = a_{m+2} p_{m+1} + p_m \textrm{ and } q_{m+2} = a_{m+2} q_{m+1} + q_m
\]
cannot be satisfied by odd values of $p_{m+1}, p_{m+2}, q_{m+1}, q_{m+2}$ and $p_m$ and $q_m$ of opposite parity.

\end{proof}

\begin{definition}\label{fdef}
For any $\alpha$ with $\lim a_k = 2$, define $f(\alpha)$ to be the eventually constant value of
\[
\frac{(p_k + p_{k+1}) \pmod 2 + \alpha ((q_k + q_{k+1}) \pmod 2)}{2}.
\]
\end{definition}
We note that by \Cref{parity}, $f(\alpha)$ is either $\frac{\alpha}{2}, \frac{1}{2}$, or $\frac{\alpha+1}{2}$. 

\begin{theorem}\label{case3}
If $\lim a_k = 2$ and $a \notin \mathbb{Z}\alpha + f(\alpha)$, then almost every $x \in X^{(\alpha, I)}$ has a shift with $3$-block Gordon structure.
\end{theorem}

\begin{proof}
By Case 1 of the proof of Theorem~\ref{onetwo}, $X^{(\alpha, I)}$ has almost-sure absence of eigenvalues if there exist infinitely many values of $k$ for which $a$ lies in a long interval at level $k$. We therefore must only show that if $a$ eventually lies in short intervals, then $a \in \mathbb{Z}\alpha + f(\alpha)$. 

Suppose that there exists $K$ so that for $k > K$, $a_{k+1} = 2$ and $a$ lies in a short interval at level $k$. Then, $\alpha$ is 
M\"{o}bius equivalent to the silver mean $\sigma = \sqrt{2} - 1$, i.e. there exist $p,q,r,s \in \mathbb{Z}$ for which $\alpha = \frac{p \sigma + q}{r \sigma + s}$ and $|ps - qr| = 1$. Then, if we assume WLOG that $q_K \alpha$ is positive, and if $a = i_K \alpha + \eta_K$ as usual, then since $S_k = \epsilon_{k-1}$, we can write (see Figure~\ref{fig6})
\[
a = i_K \alpha + \epsilon_K - \epsilon_{K+1} + \epsilon_{K+2} - \epsilon_{K+3} + \cdots
\]

\begin{figure}[h]
\centering
\begin{tikzpicture}
\draw[thick] (0,0) -- (15, 0);
\filldraw[black] (0,0) circle (3pt);
\filldraw[black] (15,0) circle (3pt);
\filldraw[red] (4,0) circle (5pt);
\node at (0, -0.5){$i_K \alpha$};
\node at (4,-0.5){a};
\node at (7.5,0.5){S};
\draw[thick] (0,-1.5) -- (15, -1.5);
\filldraw[black] (0,-1.5) circle (3pt);
\filldraw[black] (15,-1.5) circle (3pt);
\filldraw[black] (180/29,-1.5) circle (3pt);
\filldraw[red] (4,-1.5) circle (5pt);
\draw[->] (0, -1.75) -- (180/29, -1.75);
\node at (90/29,-2){$\epsilon_{K}$};
\node at (4,-2){a};
\node at (90/29,-1){S};
\node at (625/58,-1){L};
\draw[thick] (0,-3) -- (15, -3);
\filldraw[black] (0,-3) circle (3pt);
\filldraw[black] (15,-3) circle (3pt);
\filldraw[black] (180/29,-3) circle (3pt);
\filldraw[black] (360/29,-3) circle (3pt);
\filldraw[black] (105/29,-3) circle (3pt);
\filldraw[black] (285/29,-3) circle (3pt);
\filldraw[red] (4,-3) circle (5pt);
\draw[<-] (105/29, -3.25) -- (180/29, -3.25);
\node at (285/58,-3.5){$\epsilon_{K+1}$};
\node at (105/58,-2.5){L};
\node at (285/58,-2.5){S};
\node at (465/58,-2.5){L};
\node at (645/58,-2.5){S};
\node at (795/58,-2.5){S};
\draw[thick] (0,-4.5) -- (15, -4.5);
\filldraw[black] (0,-4.5) circle (3pt);
\filldraw[black] (15,-4.5) circle (3pt);
\filldraw[black] (180/29,-4.5) circle (3pt);
\filldraw[black] (360/29,-4.5) circle (3pt);
\filldraw[black] (105/29,-4.5) circle (3pt);
\filldraw[black] (285/29,-4.5) circle (3pt);
\filldraw[black] (30/29,-4.5) circle (3pt);
\filldraw[black] (210/29,-4.5) circle (3pt);
\filldraw[black] (390/29,-4.5) circle (3pt);
\filldraw[black] (135/29,-4.5) circle (3pt);
\filldraw[black] (315/29,-4.5) circle (3pt);
\filldraw[black] (60/29,-4.5) circle (3pt);
\filldraw[black] (240/29,-4.5) circle (3pt);
\filldraw[red] (4,-4.5) circle (5pt);
\draw[->] (105/29, -4.75) -- (135/29, -4.75);
\node at (120/29,-5){$\epsilon_{K+2}$};
\node at (15/29,-4){S};
\node at (45/29,-4){S};
\node at (82.5/29,-4){L};
\node at (120/29,-4){S};
\node at (315/58,-4){L};
\node at (195/29,-4){S};
\node at (225/29,-4){S};
\node at (525/58,-4){L};
\node at (300/29,-4){S};
\node at (675/58,-4){L};
\node at (375/29,-4){S};
\node at (412.5/29,-4){L};
\end{tikzpicture}
\caption{$a$ lying eventually in short intervals for $2$-splits}\label{fig6}
\end{figure}
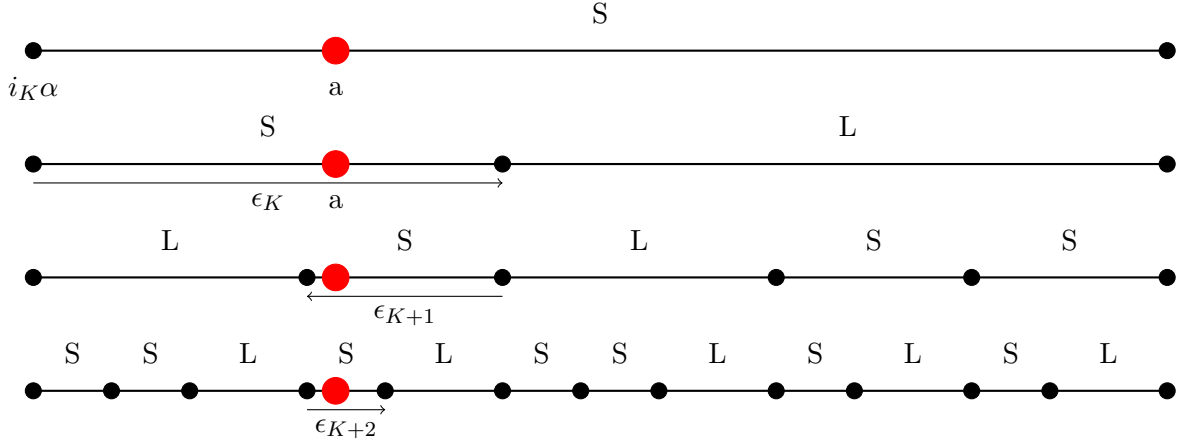
It remains to examine the infinite series $\sum_{i=K}^{\infty} (-1)^i \epsilon_i$.
Note that 
\[
s\alpha = \frac{ps\sigma + qs}{r \sigma + s} = q \pm \frac{\sigma}{r\sigma + s} \textrm { and }
r\alpha = \frac{pr\sigma + qr}{r \sigma + s} = p \pm \frac{1}{r \sigma + s}.
\]
This means that two consecutive values $\epsilon_j$ and $\epsilon_{j+1}$ have ratio $\sigma$. Since 
$\epsilon_k = 2 \epsilon_{k+1} + \epsilon_{k+2}$ for all $k > K$, this ratio is eventually $\sigma$, and by increasing $K$ if necessary, we may assume that
$\epsilon_{k} = \sigma \epsilon_{k-1}$ for all $k > K$. Then
\begin{multline*}
\epsilon_K - \epsilon_{K+1} + \epsilon_{K+2} - \epsilon_{K+3} + \cdots
= \epsilon_{K} (1 - \sigma + \sigma^2 + \cdots)
= \epsilon_{K} \frac{1}{1 + \sigma}\\
= \epsilon_{K} \frac{\sqrt{2}}{2} = \epsilon_K \frac{\sigma + 1}{2} = (\epsilon_{K+1} + \epsilon_{K})/2 = \pm \frac{q_{k+1}\alpha - p_{k+1} - q_k \alpha + p_k}{2}.
\end{multline*}
Regardless of the sign of the $\pm$, this is equal to $f(\alpha)$ plus an integer multiple of $\alpha$, and so 
since $a = i_K \alpha \pm (\epsilon_K - \epsilon_{K+1} + \epsilon_{K+2} - \cdots)$, the proof is complete.

\end{proof}

We can treat most of the $\lim a_k = 1$ cases in a similar fashion.

\begin{theorem}\label{case4}
If $\lim a_k = 1$ and $\alpha \notin \mathbb{Z}\alpha$, then almost every $x \in X^{(\alpha, I)}$ has a shift with $3$-block Gordon structure.
\end{theorem}

\begin{proof}
Suppose that $\lim a_k = 1$. By Cases 4 and 5 of the proof of \Cref{onetwo}, $X^{(\alpha, I)}$ has almost-sure absence of eigenvalues if there exist infinitely many values of $k$ for which $a$ lies in intervals which are short, then long, then long, then short, or long, then long, then long, then short. This is equivalent to infinitely many values of $x$ for which $a$ lies in long, then long, then short. We can therefore assume that this case happens only finitely many times.

A short interval becomes a long interval after a $1$-split, and so can only come from a long interval before a $1$-split.
So, if $a$ is in a short interval at some point, it is in a long interval at the previous step, and for large enough $k$, it is in a short interval at the previous step, and a long interval the step before that. If we show that the case 
long, then short, then long, then short leads to a large enough value of $\beta_k$, then we will have $\limsup \beta_k > 0$ in all cases where $a$ is in short intervals infinitely many times.

Assume that $a$ lies in a long gap at level $k$, then short at level $k+1$, then long at level $k+2$, then short at level $k+3$, and without loss of generality that $q_k \alpha$ is positive. Then, since the short interval $a$ lies in at level $k+3$ shares left endpoint with the long interval $a$ lies in at step $k$, $\eta_k < S_{k+3} = \epsilon_{k+2}$. Since $a$ lies in a long interval at level $k$, $0 \leq i_k < q_{k-1}$. 

We again argue using $-a$: we can write $-a = -i_k \alpha - \eta_k = (q_{k-1} - i_k) \alpha + (\epsilon_{k-1} - \eta_k)$. Again we can assume that $q_{k-1} - i_k > 0.0001 q_k$, and so $i_k < q_{k-1} - 0.0001q_k$. Now, consider any point $j\alpha \in Q_k$ for
$0 \leq j < 0.0001q_k$. This is the left endpoint of a long interval, and after shifting by $a$, it can be written 
as $(j + i_k)\alpha + \eta_k$. Since $j + i_k < q_{k-1}$, this is a left endpoint of a long interval shifted to the right by $\eta_k$, and so retains an interval of length at least $L_k - \eta_k > \epsilon_{k-1} + \epsilon_k - \epsilon_{k+2}
 = 2\epsilon_k + \epsilon_{k+3}$. Therefore, since there were $0.0001q_k$ such intervals,\
\[
\beta_k > \epsilon_{k+3}(q_k/10000) > q_k \epsilon_k/270000 > 1/1080000.
\]

Therefore, if $a$ lies in short intervals for infinitely many values of $k$, we have almost sure absence of eigenvalues by 
Proposition~\ref{longint}. To complete the proof, we must now only show that if $a$ eventually lies in long intervals, then $a \in \mathbb{Z}\alpha$. 

Suppose that there exists $K$ so that for $k > K$, $a_{k+1} = 1$ and $a$ lies in a long interval at level $k$. Then, $\alpha$ is 
M\"{o}bius equivalent to the golden mean $\gamma = \frac{1 + \sqrt{5}}{2}$, i.e. there exist $p,q,r,s \in \mathbb{Z}$ for which $\alpha = \frac{p \gamma + q}{r \gamma + s}$ and $|ps - qr| = 1$. Also note that for $k > K$, the length of a long interval is $\epsilon_{k-1} + \epsilon_k = \epsilon_{k-2}$. Then, if we assume WLOG that $q_K \alpha$ is positive, and if $a = i_K \alpha + \eta_K$ as usual, then we can write
\[
a = i_K \alpha + \epsilon_{K-1} - \epsilon_K + \epsilon_{K+1} - \epsilon_{K+2} + \cdots
\]
Note that 
\[
s\alpha = \frac{ps\gamma + qs}{r \gamma + s} = q \pm \frac{\gamma}{r\gamma + s} \textrm { and }
r\alpha = \frac{pr\gamma + qr}{r \gamma + s} = p \pm \frac{1}{r \gamma + s}.
\]
This means that two consecutive values $\epsilon_j$ and $\epsilon_{j+1}$ have ratio $\gamma$. Since $\epsilon_k = \epsilon_{k+1} + \epsilon_{k+2}$ for all $k > K$, this ratio is eventually $\gamma$, and by increasing $K$ if necessary, we may assume that
$\epsilon_{k-1} = \gamma \epsilon_k$ for all $k > K$. Then the above yields
\begin{multline*}
a = i_K \alpha + \epsilon_{K-1} - \epsilon_{K} + \epsilon_{K+1} - \epsilon_{K+2} + \cdots
= i_K \alpha + \epsilon_{K} (\gamma - 1 + \gamma^{-1} - \gamma^{-2} + \cdots)\\
= i_K \alpha + \epsilon_{K} \frac{\gamma}{1 + \gamma^{-1}}
= i_K \alpha + \epsilon_{K} = (i_K + q_{K})\alpha.
\end{multline*}

\end{proof}

We now treat the final cases, which follow immediately from known results.

\begin{theorem}\label{case5}
If $\lim a_n = 2$ or $\lim a_n = 1$ and $|I| \in \mathbb{Z}\alpha$, then almost every $x \in X^{(\alpha, I)}$ has a shift with $3$-block Gordon structure.
\end{theorem}

\begin{proof}
This follows from known results about subshifts generated by primitive substitutions. 
Specifically, it was shown in \cite{no3block} that for any primitive substitution $\tau$, if we denote the induced subshift by $X_\tau$, and if there is a letter $a$ and word $w$ so that $(aw)^3a = awawawa$ is in the language of $X_\tau$, then almost all points of $X_\tau$ have a shift with $3$-block Gordon structure. (This hypothesis is equivalent to having so-called repetitivity index greater than $3$, but we do not need a full definition here). 

Suppose that $\alpha$ and $I$ are as in the theorem. Then $X^{(\alpha, I)}$ is quasi-Sturmian for $\alpha$ M\"{o}bius equivalent to either $\gamma = \frac{\sqrt{5}-1}{2}$ or $\sigma = \sqrt{2}-1$. This means (by \cite{quasisturm}) that there exists a fixed primitive substitution $\pi$ so that every point of $X^{(\alpha, I)}$ is the shift of some $\pi(s)$ for $s \in S$, where $S$ is either the Sturmian according to $\gamma$ or $\sigma$. In the former case, $S$ is generated by the primitive substitution $\tau: 0 \mapsto 1, 1 \mapsto 01$, and in the latter $S$ is generated by the primitive substitution $\tau: 0 \mapsto 01, 1 \mapsto 001$.

Finally, in the former case the language of $S$ contains $awawawa$ for $a = 1$ and $w = 0110$, and in the latter case the language of $S$ contains $awawawa$ for $a = 0$ and $w = 1$. Therefore, in both cases, a positive measure subset of points of $S$ have $3$-block Gordon structure. The image of this set under $\pi$ is then a positive measure subset of $X^{(\alpha,I)}$ of points with $3$-block Gordon structure, and the proof is now complete by ergodicity.

\end{proof}

Theorems~\ref{kaminaga}, \ref{case1}, \ref{onetwo}, \ref{case3}, \ref{case4}, and \ref{case5} together complete the proof of the forward direction of Theorem~\ref{mainthm}.

\section{Hausdorff dimension of the exceptional set}

Our proof shows that for all $\alpha, a$ except the countable set described in \Cref{mainthm}, the set $T$ of points in the circle whose coding sequences do not have $3$-block Gordon structure has Lebesgue measure less than $1$ (and therefore the union of the rotations $T^c + n\alpha$ has full Lebesgue measure). The argument already involves covers of $T$ by finite unions of intervals with controlled lengths and cardinalities, and so it is natural to use it to bound the Hausdorff dimension of $T$. As a proof of concept, we here present such a proof for a large set of $\alpha$, and leave a fuller characterization to future work. We recall Theorem~\ref{HD}.

\begin{T1}
If $a_k \geq 4$ for a set of $k$ of positive upper density and $\limsup_k \left(\prod_{i=1}^k a_i\right)^{1/k} < \infty$, then $3$-block Gordon structure holds for all $x \in X^{(I,\alpha)}$ except for codings of a set 
of Hausdorff dimension strictly less than $1$.
\end{T1}

\begin{proof}
By definition of upper density, there exists $\delta > 0$ and $(k_n)$ so that at least $\delta k_n$ of the digits
$a_1, \ldots, a_{k_n}$ are greater than or equal to $4$. 
By the proof of Proposition~\ref{longint}, 
\[
T \subset \liminf (T_k + (-\epsilon_k, \epsilon_k)),
\]
where we recall that $T_k = Q_k \cup (Q_k + a)$. Since Hausdorff dimension of a countable union is the supremum of the dimensions of the sets, it is sufficient to bound the dimension of intersections $I_K = \bigcap_{k = K}^{\infty} 
T_{k} + (-\epsilon_k, \epsilon_k)$. 

We will bound the dimension of infinite intersections by considering finite intersections of the form
$\bigcap_{k=K}^{k_N} T_k + (-\epsilon_k, \epsilon_k)$ as covers. For any $N$, denote by $K \leq s_1, s_2, \ldots, s_M \leq k_N$ values so that for all
$i$, $a_{s_i + 1} \geq 4$. Note that by definition, $M \geq \delta k_N - K$. It is not hard to assume without loss of generality (by reducing $\delta$) that $s_{i+1} - s_i \geq 17$ for each $i$ and that either $a_{s_i+1} \geq 5$ or $a_{s_i + 1} = 4, a_{s_i + 2} \leq 4$ for each $i$. We first rewrite 
\[
I_K \subset \bigcap_{k = K}^{k_N} T_{k} + (-\epsilon_{k}, \epsilon_{k}) \subseteq
\bigcap_{i = 1}^{M} T_{s_i} + (-\epsilon_{s_i}, \epsilon_{s_i}) = 
\bigcup_{\omega \in \{0,1\}^{M}} \bigcap_{i = 1}^{M} (Q_{s_i} + \omega(i) a + (-\epsilon_{s_i}, \epsilon_{s_i})).
\]

We now inductively give an upper bound on the number of intervals in such intersections by considering intersections at consecutive levels $s_i$ and $s_{i+1}$. Specifically, suppose that $I$ is an interval of length $2\epsilon_{s_i}$ in 
$Q_{s_i} + \omega(i) a + (-\epsilon_{s_i}, \epsilon_{s_i})$ for some $n$. 
We now wish to consider how many intervals in $Q_{s_{i+1}} + \omega(i+1) a + (-\epsilon_{s_{i+1}}, \epsilon_{s_{i+1}})$ have nonempty intersection with $I$. 

Since $I$ has length $2\epsilon_{s_i}$ and short intervals in $Q_{s_{i+1}}$ have length $\epsilon_{s_{i+1}-1}$, 
at most $\lceil 2\epsilon_{s_i}/\epsilon_{s_{i+1}-1} \rceil$ intervals in $Q_{s_{i+1}} + \omega(i+1) a + 
(-\epsilon_{s_{i+1}}, \epsilon_{s_{i+1}})$ have nonempty intersection with $I$. Proceeding inductively, we see that 
\[
\bigcap_{i = 1}^{M} (Q_{s_i} + \omega(i) a + (-\epsilon_{s_i}, \epsilon_{s_i}))
\]
can be covered by $q_{s_1} \prod_{i=1}^{M} \lceil 2\epsilon_{s_{i-1}}/\epsilon_{s_i-1} \rceil$ intervals of length $2\epsilon_{s_{i+1}}$. By Lemma~\ref{ratios}, $\epsilon_{j-2}/\epsilon_j > a_j a_{j+1} + \frac{a_{j+1}}{a_{j+1} + 1} \geq 1.5$ for every $j$. Since 
$s_{i+1} - s_i \geq 17$, $\epsilon_{s_i}/\epsilon_{s_{i+1}-1} \geq (1.5)^8 > 25$. Therefore, for every $i$,
\[
\lceil 2\epsilon_{s_i}/\epsilon_{s_{i+1}-1} \rceil \leq 2.08 \epsilon_{s_i}/\epsilon_{s_{i+1}-1}.
\]

Therefore,
$\bigcap_{i=1}^{N} T_{s_1} + (-\epsilon_k, \epsilon_k)$ can be covered by
$2^M \left(q_{s_1} \prod_{i=1}^M 2.08 \epsilon_{s_i}/\epsilon_{s_{i+1}-1}\right)$ intervals of length $2\epsilon_{s_M}$. And,
\[
2^M \left(q_{s_1} \prod_{i=1}^M 2.08 \epsilon_{s_i}/\epsilon_{s_{i+1}-1}\right) = 
q_{s_1} 4.16^M \frac{\epsilon_{s_1-1}}{\epsilon_{s_M}} \prod_{i = 1}^M (\epsilon_{s_i}/\epsilon_{s_i-1}).
\]

By Lemma~\ref{ratios}, since $a_{s_i+1} \geq 5$ or $a_{s_i+1} = 4, a_{s_i+2} \geq 4$, we know that $\epsilon_{s_i-1}/\epsilon_{s_i} > 4.2$, and so our cover has cardinality less than $(q_{s_1} \epsilon_{s_1 - 1}) (4.16/4.2)^M (\epsilon_{s_M})^{-1}$. Recalling that $M > \delta k_N - K > \delta s_M - K$, this means that for $D = (q_{s_1} \epsilon_{s_1 - 1}) (105/104)^K$, we have a cover $C$ of $I_K$ by at most 
$D (104/105)^{\delta s_M} (\epsilon_{s_M})^{-1}$ intervals of length $2\epsilon_{s_M}$. Then, for any $\eta > 0$,
\[
\sum_{I \in C} |I|^{1-\eta} \leq 
D (104/105)^{\delta s_M} (\epsilon_{s_M})^{-1} (2\epsilon_{s_M})^{1-\eta} \leq 2D (104/105)^{\delta s_M} \epsilon_{s_M}^{-\eta}.
\]
Finally, since $\limsup_k \left(\prod_{i=1}^k a_i\right)^{1/k} < \infty$, there exists $B$ so that 
\[
\epsilon_{s_M}^{-1} = \prod_{j = 1}^{s_M} \epsilon_{j-1}/\epsilon_j < \prod_{j=1}^{s_M} (a_{j+1} + 1) < B^{s_M}.
\]
Then, if we choose $\eta =  \frac{\delta \log (105/104)}{\log B}$, then 
$2D (104/105)^{\delta s_M} \epsilon_{s_M}^{-\eta} < 2D$, meaning that $\sum_{I \in C} |I|^{1-\eta} < 2D$. Since this holds for arbitrarily fine covers $C$ (by increasing $N$), we see that $HD(I_K) \leq 1 - \eta$. Since $\eta$ is independent of $K$, this means that $HD(T) \leq 1- \eta$, completing the proof.

\end{proof}

We remark that for a set of $\alpha$ of full Lebesgue measure, the continued fraction digits are at least $4$ for a set of natural density $\log_2(5/4)$ and have geometric means approaching Khintchine's constant 
$K = 2.68545\ldots$ (see for instance \cite{khinchin}). For such $\alpha$, one can then use any $\delta > \log_2(5/4)/17$ (recalling that we had to pass to a subsequence with indices separated by $17$) and 
$B > 2K$ in the above proof, yielding the following corollary.

\begin{corollary}\label{HDcor}
There exists a set $S \subset [0,1)$ of full Lebesgue measure so that for every $\alpha \in S$ and every $a \in [0,1)$, the set of points $T$ in the circle whose codings do not have a shift with $3$-block Gordon structure has 
$HD(T) \leq 1 - \frac{\log_2(5/4) \log(105/104)}{17\log 2K} \approx 0.99989 \ldots < 1$.
\end{corollary}

This is of course quite far from optimal; the above bound could surely be improved by using the fact that higher values than $4$ appear with positive density.


\section{The silver mean case}\label{silver}

We conclude with a few more details about the interesting remaining case where we do not know whether almost-sure absence of eigenvalues hold, specifically $\lim a_n = 2$ and $|I| \in \mathbb{Z}\alpha + f(\alpha)$. First, we prove that such shifts have no points with $3$-block Gordon structure, completing the proof of \Cref{mainthm}.

\begin{theorem}\label{badcase}
When $\lim a_n = 2$ and $|I| \in \mathbb{Z}\alpha + f(\alpha)$, there exists $L$ so that $X^{(\alpha,I)}$ contains no word of the form $www$ with $|w| > L$. In particular, there are no points in $X^{(\alpha,I)}$ with $3$-block Gordon structure. 
\end{theorem}

\begin{proof}

From the proof of \Cref{case3}, we know that if
$a = |I| \in \mathbb{Z}\alpha + f(\alpha)$, then we can write
\[
a = i \alpha + (\epsilon_K - \epsilon_{K+1} + \epsilon_{K+2} - \cdots)
\]
for some very large value of $K$ with $q_K \alpha$ positive. But then, by Figure~\ref{fig6}, we can see that $a$ is in short intervals before $2$-splits for all $k > K$.

We now claim that if $a_{k+1} = a_{k+2} = 2$ and $a$ is in a short interval at steps $k$, $k+1$, and $k+2$, then 
$T_k + (-\epsilon_k, \epsilon_k) = [0,1)$.

Suppose that the above holds, and without loss of generality we assume that $q_k \alpha$ is positive. Then by Lemma~\ref{endpts}, $a = i_k \alpha + \eta_k$ for $q_{k-1} \leq i_k < q_k$. 
At step $k$, $a$ lies in the short
interval $(i_k \alpha, i_k \alpha + \epsilon_{k-1})$. 
Since $q_k \alpha$ is positive, the unique short subinterval at step $k+1$ is $(i_k \alpha, i_k \alpha + \epsilon_k)$, and since $q_{k+1} \alpha$ is negative, the unique short subinterval at step $k+1$ is 
$(i_k \alpha + \epsilon_k - \epsilon_{k+1}, i_k \alpha + \epsilon_k)$. Since $a$ lies in this interval, $\epsilon_k - \epsilon_{k+1} < \eta_k < \epsilon_k$. Therefore, $\eta_k + \epsilon_k \in (2\epsilon_k - \epsilon_{k+1}, 2\epsilon_k)$. Finally, since
$a_{k+1} = a_{k+2} = 2$,
\[
L_k - 2\epsilon_k = \epsilon_{k-1} - \epsilon_k = \epsilon_k + \epsilon_{k+1} = 2\epsilon_k - \epsilon_{k+1} - \epsilon_{k+2} < 2\epsilon_k - \epsilon_{k+1}.
\]
Therefore, 
\begin{equation}\label{oops}
\eta_k + \epsilon_k \in (L_k - 2\epsilon_k, 2\epsilon_k).
\end{equation}

By (\ref{uv}), we can write $Q_k + a = U_k \cup V_k$ where
\[
U_k = \{j \alpha + \eta_k \ : \ i_k \leq j < q_k\} \textrm{ and } V_k = \{j \alpha + (q_k \alpha + \eta_k) \ : \ 0 \leq j < i_k\}.
\]

Since all $j\alpha$ with $0 \leq j < q_{k-1}$ are left endpoints of long intervals and since $i_k \geq q_{k-1}$, all points in $V_k$ are in long intervals in $Q_k$. In particular, by (\ref{oops}), all long intervals in $Q_k$ will be split into intervals of lengths less than $2\epsilon_k$ by points in $V_k$. 

Short intervals are split by either points in $U_k$ or $V_k$. However, by \ref{oops},
$\eta_k + \epsilon_k \in (L_k - 2\epsilon_k,2\epsilon_k)$, and so
$\eta_k \in (L_k - 3\epsilon_k, \epsilon_k)$. Since $S_k = L_k - \epsilon_k)$, this means that
\[
\eta_k, \eta_k + \epsilon_k \in (S_k - 2\epsilon_k, 2\epsilon_k).
\]
Therefore, all short intervals in $Q_k$ are also split into intervals of lengths less than $2\epsilon_k$, meaning that $T_k + (-\epsilon_k, \epsilon_k) = [0,1)$ for all $k > K$ as claimed.

Now choose any $m > q_K$, and choose the unique value of $k$ for which $q_k \leq m < q_{k+1}$. it is apparent that a statement similar to Lemma~\ref{Tk} holds, i.e., if $|m \alpha| < \min(a, 1-a)$, then a coding sequence 
$c = c^{(x,I,\alpha)}$ satisfies $c(i \pm m) = c(i)$ for all $-m < i \leq 0$ iff $x$ is distance greater than $|m \alpha|$ from all points of the form $-i\alpha$ and $-i\alpha + a$ for $0 \leq i < m$. 

However, since $T_k + (-\epsilon_k, \epsilon_k) = [0,1)$, we know that every $x \in [0,1)$ is within distance $\epsilon_k$ of some point of the form
$-i\alpha$ or $-i\alpha + a$ for $0 \leq i < q_k$. Also, since $m < q_{k+1}$, $|m \alpha| \geq \epsilon_k$, and so 
$x$ is distance less than $|m \alpha|$ from the same point, with $0 \leq i < q_k \leq m$. In other words, as long as $|m \alpha| < \min(a, 1-a)$, no triple $www$ with $|w| = m$ is in the language of $X^{(\alpha, I)}$. 

It remains only to deal with the case where $|m \alpha| > \min(a, 1-a)$. For any such $m$, if we denote by $J$ the shorter of the intervals $I, I^c$, then $x \in J \Longrightarrow x + m\alpha \notin J$. This means that if we denote by $b$ the letter coding the interval $J$, then for any coding sequence $c \in X^{(I,\alpha)}$, $c(i) = b \Longrightarrow c(i+m) \neq b$. Finally, we note that by minimality, we can without loss of generality assume that $K$ is so large that 
every word $w$ of length at least $q_K$ contains a $b$, meaning that $www$ cannot possibly be in the language of $X^{(\alpha, I)}$, completing the proof.

\end{proof}

It is perhaps instructive to look in a bit more detail at some representative examples.

\begin{example}\label{ex1}
Define $\alpha = \sqrt{2}-1$ (i.e. $a_k = 2$ for all $k$) and 
$a = f(\alpha) = \frac{\alpha+1}{2} = \sqrt{2}/2$. 

This example is extremely related to a Sturmian subshift. Specifically, if we define 
$X$ to be the Sturmian induced by rotation number $a$, then it is easily checked that the silver mean example $S$ just comes from taking alternate letters from points of $X$. In other words,
\[
S = \{(x_{2n})_{n \in \mathbb{Z}} \ : \ x \in X\}.
\]

Since $X$ comes from a Sturmian with continued fraction expansion $a = [1,2,2,2,2,\ldots]$, it in fact does have almost uniform $3$-block Gordon structure! It is then natural to wonder why this does not pass through to $S$. The answer is that 
for $a$, the values of $q_k$ are given by the recursion $q_1 = 1$, $q_2 = 3$, $q_{k+1} = q_k + 2q_{k-1}$, and all $q_k$ are therefore odd. Our proofs therefore only yield triples of words $w$ of odd length, and taking every other letter of $www$ only yields a triple when $w$ has even length.
\end{example}

\begin{example}
Define $\alpha = \sqrt{2}/2$ (i.e. $a_1 = 1$ and $a_k = 2$ for all $k > 1$) and 
$a = f(\alpha) = 1/2$. This example is a so-called complementary symmetric Rote sequence, which has several nice properties. For instance, it is invariant under the coordinatewise bit-flip operation, and again there is an interesting relation to Sturmian sequences. The proof of Theorem 3 in \cite{rote} shows that a sequence $x$ is in $X$ if and only if its sequence $y$ of consecutive differences defined by $y(n) = |x(n) - x(n - 1)|$ is Sturmian for rotation number $2 \sqrt{2} - 2$.

\end{example}

Both of the above examples fail to have any $3$-block Gordon structure by \Cref{badcase}.
In systems without $3$-block Gordon structure, the most common argument for proving absence of eigenvalues is via so-called $2$-block Gordon structure (which requires only $x(i) = x(i + q_k)$ for $-q_k < i \leq 0$) and analysis of trace maps; see for instance a proof for Sturmian subshifts in \cite{sturm}. 

it is easy to check that all two-interval circle rotations have almost uniform $2$-block Gordon structure, and so the remaining step would be analysis of traces (specifically, of products of transfer matrices). This is usually proved by some sort of recursion generating words in the language; for instance, the Sturmian shift from Example~\ref{ex1} has a very simple substitutive structure coming from the recursion $w_{k+1} = w_k w_k w_{k-1}$. However, in both cases, the substitutive structure for the $2$-interval coding subshift is much more complicated and does not seem to yield a simple argument using existing techniques.

\bibliographystyle{plain}
\bibliography{2int}

\end{document}